\documentclass[reqno]{article}

\usepackage{rotating}
\usepackage{amsmath}
\usepackage{amsfonts}
\usepackage{amssymb}
\usepackage{amscd}
\usepackage{amsthm}
\usepackage{stmaryrd}
\usepackage[all]{xy}
\xyoption{line}
\SelectTips{cm}{10}
\usepackage{graphics}
\usepackage{graphicx}
\usepackage{fancyhdr}
\usepackage{enumitem}
\usepackage{mathtools}
\usepackage{url}

\theoremstyle{plain}
\newtheorem{theorem}{Theorem}[section]

\newtheorem*{slemma}{Lemma}

\theoremstyle{definition}
\newtheorem{definition}[theorem]{Definition}

\theoremstyle{remark}

\newtheoremstyle{tiny}
     {\topsep}
     {\topsep}
     {\footnotesize}
     {}
     {\itshape}
     {.}
     {.5em}
     {}

\theoremstyle{tiny}

\newcommand{\comma}{{},\,}
\newcommand{\col}{\!:\:\!}

\newcommand{\sSet}{\mathsf{sSet}}

\newcommand{\xra}[1]{\xrightarrow{#1}}

\newcommand{\ra}{\rightarrow}

\newcommand{\xlra}[1]{\xrightarrow{\ #1\ }}

\newdir{c}{{}*!/-5pt/@^{(}}
\newdir{d}{{}*!/-5pt/@_{(}}
\newdir{ >}{{}*!/-5pt/@{>}}
\newdir{s}{{}*!/+10pt/@{}}
\newdir{|>}{{}*!/2pt/@{|}*@{>}}

\newcommand{\cof}[1][]{\mathbin{\:\!\!\xymatrix@1@C=15pt{{}\ar@{ >->}[r]^{#1} & {}}}}
\newcommand{\fib}[1][]{\mathbin{\:\!\!\xymatrix@1@C=15pt{{}\ar@{->>}[r]^{#1} & {}}}}
\newcommand{\embed}[1][]{\mathbin{\:\!\!\xymatrix@1@C=15pt{{}\ar@{c->}[r]^{#1} & {}}}}

\newcommand{\pbsize}{15pt}
\newcommand{\pboffset}{.5}

\newcommand{\xycorner}[3]{\save #2="a";#1;"a"**{}?(\pboffset);"a"**\dir{-};#3;"a"**{}?(\pboffset);"a"**\dir{-}\restore}
\newcommand{\pb}{\xycorner{[]+<\pbsize,0pt>}{[]+<\pbsize,-\pbsize>}{[]+<0pt,-\pbsize>}}

\newcommand{\xymatrixc}[2][]{\xy *!C\xybox{\xymatrix#1{#2}}\endxy}

\newcommand{\xysmallbullet}{*+<.2pc>!!<0pt,\the\fontdimen22\scriptfont2>{\scriptstyle\bullet}}

\newcommand{\leftbox}[2]{\phantom{#1} \save []+L*+<.5pc>!!<0pt,\the\fontdimen22\textfont2>!L{#1#2} \restore}
\newcommand{\rightbox}[2]{\phantom{#2} \save []+R*+<.5pc>!!<0pt,\the\fontdimen22\textfont2>!R{#1#2} \restore}

\newcommand{\cleftbox}[2]{#1#2 \POS[]+L*+<.5pc>!!<0pt,\the\fontdimen22\textfont2>!L{\phantom{#1}\vphantom{#2}}+C*+<.5pc>{\phantom{#1}\vphantom{#2}}}
\newcommand{\crightbox}[2]{#1#2 \POS[]+R*+<.5pc>!!<0pt,\the\fontdimen22\textfont2>!R{\vphantom{#1}\phantom{#2}}+C*+<.5pc>{\vphantom{#1}\phantom{#2}}}

\newcommand{\stdsimp}[1]{\Delta^{#1}}
\newcommand{\vertex}[1]{#1}
\newcommand{\horn}[2]{%
\xy
<0pt,-\the\fontdimen22\textfont2>;p+<.1em,0em>:
{\ar@{-}(0,0);(3,7)},
{\ar@{-}(3,7);(5,0)},
{\ar@{-}(3.2,7);(5.2,0)},
{\ar@{-}(3.4,7);(5.4,0)}
\endxy\;\!{}^{#1}_{#2}}

\entrymodifiers={+!!<0pt,\the\fontdimen22\textfont2>}

\newcommand{\can}{\operatorname{can}}

\newcommand{\conn}{\operatorname{conn}}
\newcommand{\fibre}{\operatorname{fib}}

\newcommand{\id}{\operatorname{id}}

\newcommand{\im}{\operatorname{im}}

\newcommand{\pr}{\operatorname{pr}}

\newcommand{\then}{n}
\newcommand{\Pnew}{P_\then}
\newcommand{\Pold}{P_{\then-1}}
\newcommand{\varphin}{\varphi_\then}
\newcommand{\varphinst}{\varphi_{\then*}}
\newcommand{\varphio}{\varphi_{\then-1}}
\newcommand{\psin}{\psi_\then}
\newcommand{\psio}{\psi_{\then-1}}
\newcommand{\psinst}{\psi_{\then*}}
\newcommand{\pn}{p_\then}
\newcommand{\fn}{f_\then}
\newcommand{\fold}{f_{\then-1}}
\newcommand{\Fn}{F_\then}

\usepackage{titlesec}

\titleformat{\section}[block]
{\normalfont\Large\filcenter\bfseries}{\thesection.}{.33em}{}
\titlespacing*{\section}
{0pt}{3.5ex plus 1ex minus .2ex}{2.3ex plus .2ex}

\titleformat{\subsection}[runin]
{\normalfont\normalsize\bfseries}{\thesubsection.}{.33em}{}[.]
\titlespacing*{\subsection}
{0pt}{3.25ex plus 1ex minus .2ex}{.5em}

\title{Computing the abelian heap of unpointed stable homotopy classes of maps\thanks{%
The research was supported by the grant P201/11/0528 of the Czech Science Foundation (GA \v CR).\newline
2010 \emph{Mathematics Subject Classification}. Primary 55Q05; Secondary 55S91.\newline
\emph{Key words and phrases}. stable homotopy class, computation, heap.
}}

\author
{Luk\'a\v{s} Vok\v{r}\'{\i}nek}

\begin{document}

\maketitle

\begin{abstract}
An algorithmic computation of the set of unpointed stable homotopy classes of equivariant fibrewise maps was described in a recent paper \cite{aslep} of the author and his collaborators. In the present paper, we describe a simplification of this computation that uses an abelian heap structure on this set that was observed in another paper \cite{heaps1} of the author. A  heap is essentially a group without a choice of its neutral element; in addition, we allow it to be empty.
\end{abstract}

\section{Introduction}

This paper deals with an algorithmic computation of the set of stable homotopy classes of \emph{unpointed} maps. In the paper \cite{heaps1}, we prove in some generality that such stable homotopy classes form an abelian heap. We give a formal definition of a heap at the end of this section; on the intuitive level, it is a group without a choice of a neutral element, in very much the same way as an affine space is a vector space without a choice of a zero vector. More importantly, we allow a heap to be empty -- this may be the case with the stable homotopy classes when the target space does not admit a basepoint. Interesting examples of such spaces arise naturally in equivariant homotopy theory and fibrewise homotopy theory, e.g.\ spaces equipped with a free action of a group or fibrations that do not possess any section. We treat both cases at the same time.

\subsection*{The working category}

We denote by $G{-}\sSet$ the category of $G$-simplicial sets, i.e.\ simplicial sets equipped with a simplicial action of a finite group $G$, and $G$-maps between them (these are equivariant simplicial maps). For $X\in G{-}\sSet$ and a subgroup $H\leq G$, we denote by $X^H$ the subspace of the $H$-fixed points. Any $G$-map $\ell\colon X\to Y$ restricts to an $H$-fixed point map $\ell^H\colon X^H\to Y^H$.

Our working category $A/G{-}\sSet/B$ is that of $G$-simplicial sets under $A$ and over $B$, i.e.\ $G$-simplicial sets $X$ equipped with a pair of $G$-maps $A\to X\to B$ whose composition is a fixed $G$-map $A\to B$, surpressed from the notation. Morphisms in this category are $G$-maps $\ell\colon X\to Y$ for which both triangles in
\[\xymatrix@C=3pc{
A \ar[r]^-f \ar[d]_-\iota & Y \ar[d]^-\varphi \\
X \ar[r]_-g \ar[ru] \POS?<>(.5)+/:a(90).5pc/*{\scriptstyle\ell} & B
}\]
commute. There is also an obvious notion of homotopy (equivariant, relative to $A$ and fibrewise over $B$). In case $\iota$ is an inclusion and $\varphi$ is a $G$-Kan fibration (the $\varphi^H$ are Kan fibrations for all $H\leq G$), the resulting set of homotopy classes will be denoted by $[X,Y]^A_B$. A square as above is said to be \emph{stable} if for all subgroups $H\leq G$, the relation
\[\dim\operatorname{cof}\iota^H\leq 2\conn\operatorname{fib}\varphi^H\]
holds, where the left-hand side is the dimension of $X^H\smallsetminus A^H$ and $\conn\operatorname{fib}\varphi^H$ denotes the connectivity of the fibre of $\varphi^H$.

For general $X\comma Y$, we define $[X,Y]^A_B$ by first replacing $\iota$ up to weak homotopy equivalence by an inclusion $A\cof X^\mathrm{cof}$ and $\varphi$ by a $G$-Kan fibration $Y^\mathrm{fib}\fib B$ and then setting $[X,Y]^A_B=[X^\mathrm{cof},Y^\mathrm{fib}]^A_B$. A stable square is defined analogously in terms of this replacement.

\begin{theorem}[{\cite[Theorem~1.1]{heaps1}}]\label{t:heap_structure_fibrewise_equivariant}
Let $G$ be a fixed finite group. For any stable commutative square in $G{-}\sSet$
\[\xymatrix@=1.5pc{
A \ar[r]^-f \ar[d]_-\iota & Y \ar[d]^-\varphi \\
X \ar[r]_-g & B
}\]
the set $[X,Y]^A_B$ admits a natural structure of a (possibly empty) abelian heap.
\end{theorem}

In the special case of the $G$-action on $X$ being free, the stability requirement reads
\[\dim\operatorname{cof}\iota\leq 2\conn\operatorname{fib}\varphi,\]
since for all $H\neq 1$ we have $X^H\smallsetminus A^H=\emptyset$. This is the situation in \cite{aslep}, where we showed how to compute $[X,Y]^A_B$ under the above conditions plus $B$ being simply connected. The computation was complicated by the fact that this set has no preferred element. The algorithm of that paper can be improved by working with (natural) abelian heaps instead of (non-natural) abelian groups. This has a further advantage -- much more data can be precomputed if $Y$ is fixed; this will become apparent later.

\begin{theorem}[{\cite[Theorem~1.2]{aslep}}]\label{t:computing_homotopy_classes}
There exists an algorithm that, given a commutative square
\[\xymatrix@=1.5pc{
A \ar[r]^-f \ar[d]_-\iota & Y \ar[d]^-\varphi \\
X \ar[r]_-g & B
}\]
as in Theorem~\ref{t:heap_structure_fibrewise_equivariant}, with all $G$-actions free and with $B$ simply connected, computes the (possibly empty) abelian heap $[X,Y]^A_B$.
\end{theorem}

Since we may replace $\iota$ by an inclusion $A\to(\stdsimp 1\times A)\cup X$ in an algorithmic way, we will assume from now on that $\iota$ is an inclusion to start with. The proof of Theorem~\ref{t:computing_homotopy_classes} is given in Section~\ref{s:proof_computing_homotopy_classes}. In Section~\ref{s:short_proof}, we outline a slower but much simpler algorithm.

\subsection*{Heaps}

A \emph{Mal'cev operation} on a set $S$ is a ternary operation
\[t\colon S\times S\times S\to S\]
satisfying the following two conditions: $t(x,x,y)=y$, $t(x,y,y)=x$. It is said to be
\begin{itemize}[topsep=2pt,itemsep=2pt,parsep=2pt,leftmargin=\parindent,label=--]
\item
	\emph{asssociative} if $t(x,r,t(y,s,z))=t(t(x,r,y),s,z)$;
\item
	\emph{commutative} if $t(x,r,y)=t(y,r,x)$.
\end{itemize}
A set equipped with an associative Mal'cev operation is called a \emph{heap}. It is said to be an \emph{abelian heap} if in addition, the operation is commutative. We remark that traditionally, heaps are assumed to be non-empty. Since it is easy to produce examples where $[X,Y]=\emptyset$ in Theorem~\ref{t:heap_structure_fibrewise_equivariant}, it will be more convenient to drop this convention.

The relation of heaps and groups works as follows. Every group becomes a heap if the Mal'cev operation is defined as $t(x,r,y)=x-r+y$. On the other hand, by fixing an element $0\in S$ of a heap $S$, we may define the addition and the inverse
\[x+y=t(x,0,y),\quad -x=t(0,x,0).\]
It is simple to verify that this makes $S$ into a group with neutral element $0$. In both passages, commutativity of heaps corresponds exactly to the commutativity of groups.

\section{The Moore--Postnikov tower approach}

\subsection*{The tower}

We will first give a proof of Theorem~\ref{t:heap_structure_fibrewise_equivariant}, independent of \cite{heaps1}, under the additional assumptions of Theorem~\ref{t:computing_homotopy_classes}. We consider the \emph{Moore--Postnikov tower} of $Y$ over $B$; it consists of spaces $\Pnew$ over $B$, called the \emph{stages}, that are approximations of $Y$ in the following sense: there are factorizations
\[\xymatrix@R=.5pc@C=3pc{
& \Pnew \ar[rd]^-{\psin} \ar[dd]^-{\pn} \\
Y \ar[ru]^-{\varphin} \ar[rd]_-{\varphio} & & B \\
& \Pold \ar[ru]_-{\psio}
}\]
of $\varphi$ such that
\begin{itemize}[topsep=2pt,itemsep=2pt,parsep=2pt,leftmargin=\parindent,label=$\bullet$]
\item
	$\varphinst\colon\pi_iY\to\pi_i\Pnew$ is an isomorphism for $i<\then+1$ and an epimorphism for $i=\then+1$,
\item
	$\psinst\colon\pi_i\Pnew\to\pi_i B$ is an isomorphism for $i>\then+1$ and a monomorphism for $i=\then+1$.
\end{itemize}
Thus, $\Pnew$ looks like $Y$ in low dimensions and like $B$ in high dimensions. When both $B$ and $Y$ are simply connected, there are pullback squares (the latter in $G{-}\sSet/B$)
\begin{equation}\tag{\textsf{MPT}}\label{e:MP_stages}
\xymatrixc{
\Pnew \ar[r] \ar[d]_-{p_\then} \pb & E(\pi_\then,\then) \ar[d]^-\delta & & \Pnew \ar[r] \ar[d]_-{p_\then} \pb & B\times E(\pi_\then,\then) \ar[d]^-\delta 
\\
\Pold \ar[r]_-{k_\then'} & K(\pi_\then,\then+1) & & \Pold \ar[r]_-{k_\then} & \phantom{B\times E(\pi_\then,\then)} \POS[]+L*+<.5pc>!!<0pt,\the\fontdimen22\textfont2>!L{B\times K(\pi_\then,\then+1)}
}\end{equation}
where $K(\pi_\then,\then+1)$ is the Eilenberg--MacLane space and $E(\pi_\then,\then)$ is its path space. The standard simplicial models for these spaces are minimal and could be used for an algorithmic construction of the tower, see \cite{polypost}.

\subsection*{Fibrewise Mal'cev operations}
For $\dim(X\smallsetminus A)\leq n$, there is an isomorphism $[X,Y]^A_B\cong[X,\Pnew]^A_B$, see \cite[Theorem 3.3]{aslep}. We will show that for $\then\leq 2\conn\operatorname{fib}\varphi$, there exists a natural abelian heap structure on $[X,\Pnew]^A_B$ without any restrictions on $X$. As follows easily from Yoneda lemma, any \emph{natural} abelian heap structure on $[-,\Pnew]^A_B$ is induced by an operation
\[\tau\colon\Pnew\times_B\Pnew\times_B\Pnew\to\Pnew.\]
The result will thus follow from the following definition and theorem.

\begin{definition}
A \emph{Mal'cev operation} on $\zeta\colon Z\ra B$ is a fibrewise map
\[\tau\col Z\times_BZ\times_BZ\ra Z\]
that satisfies the Mal'cev conditions $\tau(x,x,y)=y$, $\tau(x,y,y)=x$ whenever $x$, $y$ lie in the same fibre of $\zeta$, i.e.\ it is a Mal'cev operation on $Z\in G{-}\sSet/B$.

It is said to be \emph{homotopy associative} if there exists an equivariant fibrewise homotopy
\begin{equation}\tag{\textsf{ass}}\label{e:ass}
\tau(x,r,\tau(y,s,z))\sim\tau(\tau(x,r,y),s,z)
\end{equation}
that is constant on the diagonal (i.e.~when $x=r=y=s=z$). It is said to be \emph{homotopy commutative} if
\begin{equation}\tag{\textsf{comm}}\label{e:comm}
\tau(x,r,y)\sim\tau(y,r,x)
\end{equation}
by an equivariant fibrewise homotopy that is constant on the diagonal.
\end{definition}

\begin{theorem}\label{t:heap_structure_on_MP_stages}
For each $\then\leq 2\conn\operatorname{fib}\varphi$, there exists a Mal'cev operation on $\psin\colon\Pnew\to B$ and it is unique up to homotopy. Every such operation is homotopy associative and homotopy commutative.
\end{theorem}

To finish our alternative proof of Theorem~\ref{t:heap_structure_fibrewise_equivariant}, we describe the Mal'cev operation on $[X,\Pnew]^A_B$: it is given simply as
\[t([\ell_1],[\ell_0],[\ell_2])=[\tau(\ell_1,\ell_0,\ell_2)];\]
since all maps restrict to $A$ to $\fn=\varphin f$, the Mal'cev condition gives the same for $\tau(\ell_1,\ell_0,\ell_2)$.\qed\vskip\topsep

\subsection*{Notation}

Before going into the proof of Theorem~\ref{t:heap_structure_on_MP_stages}, we introduce some notation. When $r=r_1+\cdots+r_k$ is a decomposition of a positive integer $r$, we denote by $\delta^{r_1\cdots r_k}Z$ the subspace of $Z\times_B\cdots\times_BZ$ ($r$-times) consisting of such $r$-tuples among which the first $r_1$ are equal, the next $r_2$ are equal, etc.\ and the last $r_k$ are equal. Thus $\delta^{21}Z=\{(x,x,y)\}$, $\delta^{12}Z=\{(x,y,y)\}$, etc. We also introduce the notation
\[\delta^{21,12}Z=\delta^{21}Z\cup\delta^{12}Z.\]
By this convention, $\delta^{111}Z=Z\times_BZ\times_BZ$. Thus, a Mal'cev operation on $\zeta$ is a diagonal in the following diagram, i.e.\ a map $\tau$ making both triangles commute,
\begin{equation}\tag{\textsf{Mal}}\label{e:Mal'cev}
\xy*!C\xybox{\xymatrix@C=3pc{
\delta^{21,12}Z \ar[r]^-\can \ar@{ >->}[d] & Z \ar@{->>}[d] \\
\delta^{111}Z \ar[r] \ar@{-->}[ur]^-\tau & B
}}\endxy\end{equation}
where $\can\colon\delta^{21,12}Z\to Z$ is the map sending $(x,x,y)\mapsto y$ and $(x,y,y)\mapsto x$ as in the definition of a Mal'cev operation.

\subsection*{Proof of Theorem~\ref{t:heap_structure_on_MP_stages}}

Denoting $d=\conn\operatorname{fib}\varphi$, the homotopy groups of the fibre $\Fn=\fibre\psin$ of $\psin\colon\Pnew\to B$ are concentrated in dimensions $d+1$ through $2d$. Thus, the obstructions to the existence of a lift $\tau$ as in \eqref{e:Mal'cev} lie in cohomology groups
\[H^{i+1}_G(\delta^{111}\Pnew,\delta^{21,12}\Pnew;\pi_i\Fn)\]
for $d+1\leq i\leq 2d$. Similarly to the proof of \cite[Theorem~5.3]{aslep}, it can be shown that projections from both $\delta^{111}\Pnew$ and $\delta^{21,12}\Pnew$ to their middle component $\Pnew$ are quasi-fibrations with fibres as indicated in the left of the following diagram
\[\xymatrix{
\Fn\vee \Fn \ar[r] \ar[d] &
	\delta^{21,12}\Pnew \ar[r] \ar[d] &
	\Pnew \ar[d]^-\id \\
\Fn\times \Fn \ar[r] &
	\delta^{111}\Pnew \ar[r] &
	\Pnew
}\]
Since the pair $(\Fn\times \Fn,\Fn\vee \Fn)$ is known to be $(2d+1)$-connected, the map of the fibres is a $(2d+1)$-equivalence and thus, so is the map of the total spaces. In particular, the above cohomology groups vanish for $i\leq 2d$.

Each obstruction to the existence of a homotopy between two diagonals $\tau$, $\tau'$ as above lies one dimension lower and is thus also zero, proving the uniqueness up to homotopy.

Concerning the homotopy associativity, the two sides of \eqref{e:ass} prescribe two maps $\delta^{11111}\Pnew\to \Pnew$ that restrict to the same map on $\delta^{2111,1112}\Pnew$, i.e.\  when either $x=r$ or $s=z$. Thinking of both spaces as spaces over the middle three components $\delta^{111}\Pnew$, the fibres are again $\Fn\vee \Fn$ and $\Fn\times \Fn$. Thus, the respective obstructions vanish and the two sides of \eqref{e:ass} are fibrewise homotopic relative to $\delta^{2111,1112}\Pnew$ which contains $\delta^5\Pnew$ as desired. The proof of homotopy commutativity is similar.\qed\vskip\topsep

\section{Constructing Mal'cev operations}

We will now turn our attention to the algorithmic side. We will describe a way of constructing Mal'cev operations on Moore--Postnikov stages $\Pnew$. It turns out that in order to carry out this construction algorithmically, we need to weaken the Mal'cev conditions.

\begin{definition}
A \emph{weak Mal'cev operation} on a fibration $\zeta\colon Z\fib B$ is a map
\[\tau\colon Z\times_BZ\times_BZ\ra Z\]
together with equivariant fibrewise homotopies
\[\lambda\col y\sim\tau(x,x,y),\quad \rho\col x\sim\tau(x,y,y)\]
and a ``second order'' equivariant fibrewise homotopy $\eta$ of the restrictions of $\lambda$ and $\rho$ to the diagonal:
\[\xymatrix@=1.5pc{
& \tau(x,x,x) \POS[];[d]**{}?<>(.5)*{\scriptstyle//\eta_x//} \\
x \ar[ru]^-{\lambda_{x,x}} \ar[rr]_-{s_0x} & & x \ar[lu]_-{\rho_{x,x}}
}\]
\end{definition}

Again, it is possible to formalize weak Mal'cev operations as diagonals in a diagram similar to \eqref{e:Mal'cev}, namely
\begin{equation}\tag{\textsf{wMal}}\label{e:weak_Mal'cev}
\xy*!C\xybox{\xymatrix@C=3pc{
\widehat\delta^{21,12}\Pnew \ar[r] \ar@{ >->}[d] & \Pnew \ar@{->>}[d] \\
\widehat\delta^{111}\Pnew \ar[r] \ar@{-->}[ur]^-\tau & B
}}\endxy\end{equation}
where $\widehat\delta^{111}\Pnew$ is a subspace of $\stdsimp 2\times\delta^{111}\Pnew$ given as
\[\widehat\delta^{111}\Pnew=(\stdsimp 2\times\delta^3\Pnew)\cup(d_1\stdsimp 2\times\delta^{21}\Pnew)\cup(d_0\stdsimp 2\times\delta^{12}\Pnew)\cup(2\times\delta^{111}\Pnew)\]
and $\widehat\delta^{21,12}\Pnew$ is the part over the second face $d_2\stdsimp 2$, i.e.
\[\widehat\delta^{21,12}\Pnew=(d_2\stdsimp 2\times\delta^3\Pnew)\cup(0\times\delta^{21}\Pnew)\cup(1\times\delta^{12}\Pnew).\]
For a similar construction, consult \cite[Section~5.5]{aslep}.

\begin{theorem}\label{t:constructive_heap_structure_on_MP_stages}
There is an algorithm that, given a map $\varphi\colon Y\to B$ between simply connected finite simplicial sets and $n\leq 2\conn\operatorname{fib}\varphi$, constructs a weak Mal'cev operation on the $n$-th stage $\psi_n\colon\Pnew\to B$ of the Moore--Postnikov tower of $Y$ over $B$.
\end{theorem}

\subsection*{Notation}

We denote the standard $n$-dimensional simplex by $\stdsimp n$, its boundary by $\partial\stdsimp n$, its $i$-th face by $d_i\stdsimp n$ and its $i$-th vertex by $\vertex i$. We also use $I=\stdsimp 1$.

\subsection*{The induced Mal'cev operation on homotopy classes}

Before going into the proof of Theorem~\ref{t:constructive_heap_structure_on_MP_stages}, we will show how this result yields a new proof of Theorem~\ref{t:computing_homotopy_classes}. The problem with the weak Mal'cev operation $\tau$ is that if all $\ell_1\comma\ell_0\comma\ell_2\colon X\to\Pnew$ restrict to $A$ to a given map $f_n=\varphin f$, the composition $\tau(\ell_1,\ell_0,\ell_2)$ will not. In particular, it does not induce a Mal'cev operation in $[X,\Pnew]^A_B$ in a straightforward way. For that, we will need a strictification result.

\begin{slemma}
Every weak Mal'cev operation on a fibration $\zeta\colon Z\ra B$ is fibrewise homotopic to a strict Mal'cev operation.
\end{slemma}

\begin{proof}
The weak Mal'cev operation $\tau$ yields a diagram
\[\xymatrix@C=3pc{
\widehat\delta^{111}Z\cup(\partial_2\stdsimp{2}\times\delta^{21,12}Z) \ar[r]^-{(\tau,\can)} \ar@{ >->}[d]_-\sim &
	Z \ar@{->>}[d]^-\zeta \\
\stdsimp{2}\times\delta^{111}Z \ar[r] \ar@{-->}[ru] &
	B
}\]
The indicated diagonal exists since the map on the left is a weak homotopy equivalence (the proof is similar to \cite[Lemma 7.2]{aslep}); its restriction to $0\times\delta^{111}Z$ is a Mal'cev operation $\tau'$.
\end{proof}

Thus, we may define $t([\ell_1],[\ell_0],[\ell_2])=[\tau'(\ell_1,\ell_0,\ell_2)]$. From the computational point of view, this is not very satisfactory since $\tau'$ is hard to compute.%
\footnote{
	In terms of \cite{aslep}, this would require that the pair formed by the spaces on the left of the diagram in the lemma admits effective homology. We do not know whether this is true.}
However, the composition $\tau'(\ell_1,\ell_0,\ell_2)$ may be computed (up to homotopy) in the following way:

Compose the left cancellation homotopy $\lambda_{x,y}$ with $\fn\colon A\to\Pnew$ in both variables $x$, $y$ to obtain a homotopy $\fn\sim\tau(\fn,\fn,\fn)$ of maps $A\to\Pnew$. Extend this arbitrarily to an equivariant fibrewise homotopy $\ell\sim\tau(\ell_1,\ell_0,\ell_2)$; it is simple to prove that the resulting homotopy class $[\ell]\in[X,\Pnew]^A_B$ is independent of the choice of an extension.

Moreover, $\tau'(\ell_1,\ell_0,\ell_2)$ is also obtained in this way -- the restriction of the diagonal in the proof of the previous lemma to the first boundary $d_1\stdsimp 2\times\delta^{111}Z$ is a homotopy $\tau'\sim\tau$; then compose this with $(\ell_1,\ell_0,\ell_2)$. As a consequence, $[\ell]=t([\ell_1],[\ell_0],[\ell_2])$. Since extending homotopies defined on a pair $(X,A)$ of finite simplicial sets and lifting them in the Moore--Postnikov tower is possible in an algorithmic way by the results of \cite[Section~3.4]{aslep}, one may compute the Mal'cev operation in $[X,\Pnew]^A_B$ if the weak Mal'cev operation on $\psi_n\colon\Pnew\to B$ has been computed.

\subsection*{Proof of Theorem~\ref{t:computing_homotopy_classes}} \label{s:proof_computing_homotopy_classes}

Construct the Moore--Postnikov tower using \cite{aslep} and weak Mal'cev operations on all stages $\Pnew$ with $n\leq\dim_AX$ using Theorem~\ref{t:constructive_heap_structure_on_MP_stages}. We stress here that this part is independent of $A$ and $X$ and may be precomputed when $\varphi\colon Y\to B$ is fixed.

We denote $L_\then=B\times K(\pi_\then,\then)$ and $K_{\then+1}=B\times K(\pi_\then,\then+1)$. Assuming that $[X,\Pold]^A_B$ is non-empty, we consider the following heaps
\begin{equation}\tag{\textsf{les}}\label{e:les}
[\Sigma X,\Pold]^{\Sigma A}_B\xlra{\partial}[X,L_n]^A_B\xlra{a}[X,\Pnew]^A_B\xlra{p_{n*}}[X,\Pold]^A_B\xlra{k_{n*}}[X,K_{n+1}]^A_B
\end{equation}
(they will be made into an exact sequence later). The following are easily verified:
\begin{itemize}[topsep=2pt,itemsep=2pt,parsep=2pt,leftmargin=\parindent,label=$\bullet$]
\item
	$p_{n*}$ and $k_{n*}$ are heap homomorphisms (the later is obtained in the course of the proof of Theorem~\ref{t:constructive_heap_structure_on_MP_stages} below as the claim that $m$ factors through the contractible $E(\pi_\then,n)$),
\item
	$\im p_{n*}=(k_{n*})^{-1}(0)$;
\end{itemize}
the action of $K(\pi_n,n)$ on $\Pnew$ induces an action of $[X,L_n]^A_B$ on $[X,\Pnew]^A_B$ which
\begin{itemize}[topsep=2pt,itemsep=2pt,parsep=2pt,leftmargin=\parindent,label=$\bullet$]
\item
	is transitive on each fibre of $p_{n*}$.
\end{itemize}

To make sense of the first term and the map $\partial$, we have to assume that $[X,\Pnew]^A_B$ is non-empty and fix $[\ell_0]\in[X,\Pnew]^A_B$. The first term is interpreted%
	\footnote{The suspensions $\Sigma A$, $\Sigma X$ are meant to be taken in the category $G{-}\sSet/B$ of spaces over $B$, see \cite{heaps1}. However, it is easy to see that we obtain the same result if they are interpreted in $G{-}\sSet/X$, in which case we get $\Sigma X=I\times X$ and $\Sigma A$ is the indicated subspace.}
as $[I\times X,\Pold]^{(\partial I\times X)\cup(I\times A)}_B$, the homotopy classes of maps $I\times X\to\Pold$ fixed on the indicated subspace in the following way: the restriction to both ends $i\times X$, $i=0\comma 1$, is the composition $\ell_0'=p_n\ell_0$ and the restriction to $I\times A$ is the composition $I\times A\xra\pr A\xra{\fold}\Pold$.

The map $\partial$ is defined as follows: given $h\colon I\times X\to\Pold$, lift it along $p_n\colon\Pnew\to\Pold$ to a homotopy $\widetilde h$ starting at $\ell_0$ and such that the restriction to $I\times A$ is the composition $I\times A\xra\pr A\xra{\fold}\Pold$. Then the restriction $\widetilde h_\mathrm{end}=\widetilde h|_{1\times X}$ lies over $\ell_0'$. Since $p_n$ is a principal twisted cartesian product with structure group $K(\pi_\then,\then)$, there exists a unique $z\colon X\to L_n$ such that $\widetilde h_\mathrm{end}=\ell_0+z$ and we set $\partial[h]=[z]$. It is then easily verified that
\begin{itemize}[topsep=2pt,itemsep=2pt,parsep=2pt,leftmargin=\parindent,label=$\bullet$]
\item
	$\partial$ is a group homomorphism,
\item
	the stabilizer of $[\ell_0]\in[X,\Pnew]^A_B$ is the image of $\partial$.
\end{itemize}

Assuming that some $[\ell_0]\in[X,\Pnew]^A_B$ has been computed, an actual exact sequence of abelian groups is obtained by defining $a$ in \eqref{e:les} to be $a[z]=[\ell_0+z]$; then $[X,\Pnew]$ is computed as in \cite{aslep}.

Thus, it remains to decide whether $[X,\Pnew]^A_B$ is non-empty and if this is the case, compute an element of this heap. Since $k_{n*}$ is a heap homomorphism, it is simple to decide whether $0$ lies in its image -- namely, denoting the heap generators of $[X,\Pold]^A_B$ by $x_1\comma\ldots\comma x_r$, we are looking for a solution of the system
\[t_1k_{n*}(x_1)+\cdots+t_rk_{n*}(x_r)=0,\quad t_1+\cdots+t_r=1\]
in the abelian group $[X,K_{n+1}]^A_B$. This is easy using the Smith normal form (see e.g.\ \cite{stable_maps,aslep}). Once the coefficients $t_i$ are determined, we compute
\[[\ell_0']=t_1x_1+\cdots+t_rx_r\in(k_{n*})^{-1}(0)\]
using the Mal'cev operation in $[X,\Pold]^A_B$. A lift $\ell_0$ of $\ell_0'$ is computed by \cite[Proposition~3.5]{aslep}.\qed\vskip\topsep

\subsection*{Proof of Theorem~\ref{t:constructive_heap_structure_on_MP_stages}}

Using \eqref{e:MP_stages}, we may think of $\Pnew$ as a subset of $\Pold\times E(\pi_\then,\then)$. Assuming that a weak Mal'cev operation has been constructed on $\Pold$, we define it on $\Pnew$ by the formula
\[\tau((x,c),(r,e),(y,d))=(\tau(x,r,y),c-e+d+M(x,r,y)),\]
where $M$ is a map $\widehat\delta^{111}\Pold\to E(\pi_\then,\then)$ yet to be specified. The right-hand side lies in the pullback if and only if $k_n\tau(x,r,y)=\delta c-\delta e+\delta d+\delta M(x,r,y)$, i.e.~$M$ should be the coboundary of
\[m(x,r,y)=k_n\tau(x,r,y)-(\delta c-\delta e+\delta d)=k_n\tau(x,r,y)-(k_nx-k_nr+k_ny),\]
the deviation of $k_n$ from being a heap homomorphism. It is rather straightforward to add higher coherence data to $M$ (see \cite[Section~5.11]{aslep}) and obtain the following diagram
\[\xymatrix@C=3pc{
\widehat\delta^{21,12}\Pold \ar[r]^-0 \ar@{ >->}[d] & E(\pi_\then,\then) \ar@{->>}[d]^-\delta \\
\widehat\delta^{111}\Pold \ar[r]_-m \ar@{-->}[ur]^-M & K(\pi_\then,\then+1)
}\]
in which a diagonal exists and can be computed as in \cite{aslep}.\qed\vskip\topsep

\section{A short proof of Theorem~\ref{t:computing_homotopy_classes} using suspensions}\label{s:short_proof}

We present an alternative proof of Theorem~\ref{t:computing_homotopy_classes}. As we will see, its major flow is that we have to compute an additional Moore--Postnikov stage $P_n$ with $n=1+\dim X$. Since the computation of $P_n$ from $(n,\varphi\colon Y\to B)$ is \#P-hard even for $n$ encoded in unary (see \cite[Theorem~1.2]{hardness}), we believe that this will lead to an unpractical algorithm. On the other hand, the description of the algorithm is much simpler.

According to the proof of the main result in \cite{heaps1}, there is a bijection
\[[X,Y]^A_B\cong[\Sigma_BX,\Sigma_BY]^{\Sigma_BA}_B=[I\times X,\Sigma_BY]^{(\partial I\times X)\cup(I\times A)}_B,\]
where the fibrewise suspension $\Sigma_BY$ is obtained from $I\times Y$ by separately squashing each of $0\times Y$ and $1\times Y$ to $B$ using the given projection $\varphi\colon Y\to B$; it is naturally a space over $B$. We think of it as a space under $(\partial I\times X)\cup(I\times A)$ via the maps
\[\xymatrix@R=1pc{
\partial I\times X \ar[r]^-{\id\times g} & \partial I\times B \ar@{c->}[r] & \Sigma_BY \\
I\times A \ar[r]^-{\id\times f} & I\times Y \ar[r]^-{\mathrm{can}} & \Sigma_BY.
}\]

Replacing $\Sigma_BY\to B$ by its $n$-th Moore--Postnikov stage, $n=1+\dim X$, we are thus left to compute $[I\times X,P_n]^{(\partial I\times X)\cup(I\times A)}_B$. Again, we have an exact sequence \eqref{e:les}, with heap structures coming from the suspension this time. The argument of the proof of Theorem~\ref{t:computing_homotopy_classes} works equally well, once we provide an algorithm computing the Mal'cev operation in $[I\times X,P_n]^{(\partial I\times X)\cup(I\times A)}_B$. Given three maps $\ell_1\comma\ell_0\comma\ell_2\colon I\times X\to P_n$, we organize them into a single map
\[(e_{02}\times X)\cup(e_{12}\times X)\cup(e_{13}\times X)\xlra{(\ell_1,\ell_0,\ell_2)}P_n,\]
where $e_{ij}$ denotes the edge in $\stdsimp 3$ with vertices $i$ and $j$.
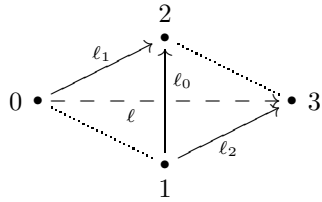
\[\xy
<0pc,0pc>;<2pc,0pc>:
(-2,1)*+{\scriptstyle\bullet}="a0"*++!R{0},"a0";
(0,2)*+{\scriptstyle\bullet}="a2"*++!D{2},"a2"**\dir{-};?>*\dir{>},?<>(.5)*+!D{\scriptstyle\ell_1},
(0,0)*+{\scriptstyle\bullet}="a1"*++!U{1},"a1"**\dir{-};?<*\dir{<},?<>(.3)*+!L{\scriptstyle\ell_0},
(2,1)*+{\scriptstyle\bullet}="a3"*++!L{3},"a3"**\dir{-};?>*\dir{>},?<>(.5)*+!U{\scriptstyle\ell_2},
"a0";"a3"**\dir{},?<>(.35)*+!U{\scriptstyle\ell},?<>(.45)**\dir{--},"a0";"a3"**\dir{};?<>(.55);**\dir{--},?>*\dir{>},
"a0";"a1"**\dir{.},"a2";"a3"**\dir{.}
\endxy\]
Together with the composition $\stdsimp 3\times A\xlra{\pr} A\xlra{f_n}P_n$, these describe the top map in the diagram
\[\xymatrix@C=3pc{
((e_{02}\cup e_{12}\cup e_{13})\times X)\cup(\stdsimp 3\times A) \ar[r] \ar@{ >->}[d]_-\sim &
	P_n \ar@{->>}[d]^-{\psi_n} \\
\stdsimp{3}\times X \ar[r] \ar@{-->}[ru] &
	B
}\]
A diagonal can be computed using \cite[Proposition~3.5]{aslep}; its restriction to $e_{03}\times X$, denoted by $\ell$ in the above picture, gives a representative of $t([\ell_1],[\ell_0],[\ell_2])$.\qed

\vskip 20pt
\vfill
\vbox{\footnotesize%
\noindent\begin{minipage}[t]{0.5\textwidth}
{\scshape
Luk\'a\v{s} Vok\v{r}\'inek}
\vskip 2pt
Department of Mathematics and Statistics,\\
Masaryk University,\\
Kotl\'a\v{r}sk\'a~2, 611~37~Brno,\\
Czech Republic
\vskip 2pt
\url{koren@math.muni.cz}
\end{minipage}
}

\end{document}